\documentclass[a4paper,twoside]{article}

\usepackage[english]{babel}
\usepackage[utf8x]{inputenc}
\usepackage{amsmath,amssymb,amsthm}%
\usepackage{eucal}
\usepackage[dvipsnames]{xcolor}
\usepackage{graphicx}
\usepackage{placeins}

\usepackage{cite}

\usepackage[margin=1in]{geometry}

\renewcommand*{\d}{\mathop{}\!\mathrm{d}}
\def\pp{\partial}

\def\a{\alpha}
\def\b{\beta}
\def\de{\delta}
\def\s{\sigma}
\def\S{\Sigma}
\def\l{\lambda}
\def\L{\Lambda}
\def\g{\gamma}

\def\f{\frac}
\def\ra{\rightarrow}

\def\Prb{\mathbf{P}}
\def\Mean{\mathbf{E}}

\newcounter{for}[section]
\newcommand{\be}[1][]{\addtocounter{for}{1} \begin{equation}\label{#1}}
\def\ee{\end{equation}}

\newcommand{\bk}[1]{\ensuremath{\langle#1\rangle}}

\newcommand{\wt}[1]{\tilde{#1}}
\renewcommand{\vec}[1]{\ensuremath{\boldsymbol{#1}}}

\newcommand{\m}[1]{\mathcal{#1}}
\newcommand{\mbb}[1]{\mathbb{#1}} 

\DeclareMathOperator{\Var}{\mathrm{Var}}

\theoremstyle{definition}
\newtheorem{rem}{remark}[section]
\newtheorem{theorem}{Theorem}[section]

\numberwithin{equation}{section}

\title{\bf A Curie-Weiss model with dissipation}

\author{Paolo Dai Pra, Markus Fischer, Daniele Regoli\\[1ex]
	Dipartimento di Matematica,
	Universit\`a degli Studi di Padova,\\
	via Trieste 63,
	I-35121 Padova,
	Italy\\
	e-mail: daipra@math.unipd.it, fischer@math.unipd.it, regoli@math.unipd.it
}

\date{\today}

\begin{document}

\maketitle

\begin{abstract}
We consider stochastic dynamics for a spin system with mean field interaction, in which the interaction potential is subject to noisy and dissipative stochastic evolution. We show that, in the thermodynamic limit and at sufficiently low temperature, the magnetization of the system has a time periodic behavior, despite of the fact that no periodic force is applied.
\end{abstract}




\section{Introduction}

Dynamics of stochastic systems with many interacting components are often described in terms of the {\em interaction energy} $H(x) = H(x_1,x_2,\ldots,x_N)$, by stochastic equations of the form
\be[1]
\d X_t = - \nabla H(X_t)\d t + \s \d B_t,
\ee
where $\s$ is the diffusion constant and $B$ is a Brownian Motion. The collective $N \ra +\infty$ behavior depends on the {\em entropy-energy balance}:
the stochastic term $\s \d B_t$ in \eqref{1} is a source of disorder, and it competes with the {\em drift} $- \nabla H(X_t)$ which tends to ``freeze'' the system in the states of minimal energy.

This structure, which admits various versions with discrete time and/or space, has proved to be successful in many applications, starting from physics but then developing in other fields such as Economics, Biology and Engineering. 

A substantial modification of the structure above consists in introducing an interaction energy which is not a {\em given} deterministic function of the state, but it has its own (possibly random) evolution. 
Models of this type have been recently proposed, independently, in various fields of research, but in particular in Biology and Economics. A first example (see e.g. \cite{Schweitzer}) is provided by the following modification of \eqref{1}, that we write componentwise:

\begin{eqnarray} \label{2}
\d X^i_t & = & \pp_x h(X^i_t, t) \d t + \s \d B^i_t \\ \label{3}
\pp_t h(x,t) & = & - \a h(x,t) + D \pp^2_x h(x,t) + \b \sum_{i=1}^N g(X^i_t,x).
\end{eqnarray}

This system has arisen, for instance, in cellular dynamics where cells tend to move toward positions $x$ of high concentration $h(x,t)$ of a given chemical (equation \eqref{2}).
This chemical is subject to dissipation ($\a >0$) and diffusion ($D>0$). Moreover, while moving, cells release themselves chemicals into the medium, thus modifying the concentration $h(x,t)$. 
It should be noticed that the model is of {\em mean field} type, in the sense that it is left invariant (in distribution) by permutations of cells. 
This is one of the instances where mean field models, which in physics are often regarded as toy models, play a more substantial role. 
A discrete-time version of this model has been recently studied in \cite{DelMoral}.

Cellular dynamics provides other examples of interacting systems. 
A very active field of research is that of synchronization in synthetic biology (see e.g. \cite{garcia2004modeling, mcmillen2002synchronizing, prindle2011sensing}). 
Concentrations of certain molecules within cells  vary over time due to chemical reactions, and may lead to periodic behavior. If $x$ denotes a vector of concentrations referred to a single cell, its evolution can be described by an ordinary differential equation 
\be[4]
\dot{x} = F(x)
\ee
which exhibits limiting cycles. Given a system of $N$ cells, various interaction mechanisms have been either observed or reproduced in experiments. 
As shown in \cite{garcia2004modeling}, suitable molecules can be ``pumped'' into the extracellular space, reach through the cell membrane the intracellular space and influence the dynamics of $x$. 
Denoting by $x_i$ the concentration vector of the $i$-th cell, by $S$ the extracellular concentration of the injected substance and by $s_i$ its concentration within cell $i$, \eqref{4} is modified as follows
\be[5]
\begin{array}{rcl}
\dot{x_i} & = & F(x_i) + g(s_i) \\
\dot{s_i} & = &  \varphi(x_i) - \b s_i + \g (S-s_i) \\
\dot{S} & = & - \a S + k(\overline{s} - S),
\end{array}
\ee
where $\overline{s} := \frac{1}{N} \sum_i s_i$. 
Random noise can be added to these equations: in all cases the interaction is driven by an ``external'' variable $S$, which is subject to dissipation ($\a>0$), and it is of mean-field type.

A third example we mention has been proposed in \cite{giesecke2011default}, it is related to finance and it models defaults in a network of firms. For $i=1,2,\ldots,N$, let $\s_i(t) \in \{0,1\}$ be the indicator of the $i$-th firm; in other words, $\s_i(t)= 0$ means that the $i$-th firm has defaulted. 
The dynamics is assigned by giving the default rates (probability of default per unit time) $\l_i(t)$. To model {\em clustering} of defaults, one can choose $\l_i = \l_i(\s_1,\ldots,\s_N)$ to be a {\em deterministic} function of $\s_1,\ldots,\s_N$, {\em decreasing} with respect to the natural partial order in $\{0,1\}^N$ (see e.g.\cite{DPRS}); in this way, the default of one firm {\em permanently} increases the rate of default of other firms. 
As observed in \cite{giesecke2011default}, the permanence of this effect is unrealistic, and should be subject to a sort of dissipation. They propose the following stochastic dynamics for the default rates:
\be[6]
\d\l_i(t) = -\a(\l_i(t) - \bar{\l}_i)\d t + \b \d L_N(t) + \s \sqrt{\l_i(t)} \d B_i(t) + \g \l_i(t) \d X(t),
\ee
where $ \bar{\l}_i \geq 0$ are given reference rates,  $L_N(t) := \frac{1}{N} \sum_i \s_i(t)$, the $B_i$'s are independent Brownian motions, and $X$ is some driving {\em exogenous factor}, such as a macroeconomic index. If $ \bar{\l}_i $ does not depend on $i$, this is again a permutation invariant model, subject to dissipation.

In terms of rigorous analysis, some results concerning the above models are available. 
In particular a macroscopic equation, which describes the limiting behavior in the limit $N \ra +\infty$, has been derived for a discrete-time version of \eqref{2}, \eqref{3}, together with sufficient conditions for this macroscopic equation to have a globally stable fixed point \cite{DelMoral}.     
For the system in \eqref{6}, a macroscopic equation has also been derived, together with a Large Deviation analysis of the $N \ra +\infty$ limit.

In comparison with what has been obtained for more traditional models motivated by statistical mechanics, one aspect that is missing is the detailed analysis of the {\em phase diagram} for some {\em special model}. 
This could reveal, for instance, what effects dissipation has in the long-time behavior of the system. 
The aim of this paper is to perform this analysis for a simple model that, however, exhibits some of the main features of the models above: the interaction is of mean-field type, and it is subject to dissipation.  
In absence of dissipation, this model reduces to (a slight modification of) the well known Curie-Weiss model, whose macroscopic properties are well understood for both the statics and the dynamics. In particular, for sufficiently low temperature, the system self-organizes, in the sense that spins tend to align producing a macroscopic magnetization that becomes constant in the long-time limit. 
We will see that in presence of dissipation this picture changes: self-organization is still present, but the macroscopic 
magnetization fluctuates periodically  driven by a two-dimensional ordinary differential equation similar to that of the classical Van der Pol oscillator, in which the large damping regime correspond to low dissipation in our model.
Despite of the simplicity of the model, some analytic aspects are considerably harder than in the standard Curie-Weiss model.

The emergence of {\em self sustained} periodic behavior (i.e. not induced by external periodic forces) has been recently studied in several models. In particular, the neural networks in \cite{perth1,perth2} and the {\em active rotators} in \cite{GB} appear to be close in spirit to the simpler model we propose. Precise analogies emerge, in particular, with the active rotators model. At macroscopic level, the ``inactive'' system has a one dimensional, compact stable manifold of  fixed points. Introducing a small forcing term, a slow, possibly periodic motion on the slightly deformed stable manifold is introduced. In our model, in absence of dissipation and with a suitable choice of variables, fixed points form a one dimensional, non-compact manifold, with an unstable part and two stable branches. A small dissipation introduces a slow motion of the stable part of the manifold; as the unstable part is met, the motion is driven to the other stable branch. This scheme is then repeated, producing periodic motion.

We finally mention that in \cite{lindner2004}, various models for noise-induced behavior are presented; in partucular, the model we propose has several analogies with the FitzHugh-Nagumo model \cite{fitzhugh1961,nagumo1962}.

In Section 2 we define the model, and derive its macroscopic evolution equation. In Section 3 we study the long-time behavior of the macroscopic equation in the case the evolution of the spin-flip rates is dissipative but not driven by additional noises. Numerical simulations for a more general case are given in Section 4.

\section{Microscopic and macroscopic evolutions}\label{sec:micromacro}

\subsection{The microscopic model}

Let $s = (s_1,s_2,\ldots,s_N) \in \{-1,1\}^N$ be a configuration of $N$ spins. 
We define, to begin with at an informal level, a stochastic process $(s(t))_{t \geq 0}$ by assigning (besides an initial condition) the {\em spin flip rates}. 
For $s \in \{-1,1\}^N$ and $i = 1,2,\ldots,N$, define $s^i$ to be the configuration obtained from $s$ by flipping $s_i$ (i.e. $s^i_i = - s_i$) and with all other spins left unchanged. 
At a given time $t \geq 0$, if $s(t) = s$, then each transition $s_i \ra -s_i$ occurs at rate $1+\tanh(s_i\l_i)$, where $\l_i(t)$ is itself a stochastic process, which evolves according to the stochastic differential equation
\be[7]
\d\l_i(t) = - \a \l_i(t)\d t + \s \d B_i(t) + \d \phi(s(t)),
\ee
where $\a, \s  \geq 0$, 
\be[mean_field]
\phi(s) := -\b m_N(s) := -\frac{\b}{N} \sum_{k=1}^N s_k,
\ee
with $\b \geq 0$, and $B_1,B_2,\ldots,B_N$ are independent Brownian motions. 
A different $\phi$ could be chosen to generate other models; we concentrate on this special choice. Note that between two consecutive spin flips, the $\l_i$'s evolve as independent Ornstein-Uhlenbeck processes; at a spin flip time $t$, each $\l_i$ jumps by a quantity of $\phi(s(t)) - \phi(s(t-))$. 

This construction can be made rigorous. In particular, it defines a
 {\em Markov process} $(s(t),\l(t)) \in \{-1,1\}^N \times \mbb{R}^N$ whose {\em infinitesimal generator} is given by
\begin{multline}\label{8}
\m{L}^N f(s,\l) := \sum_{i=1}^N \left[ (1+\tanh(s_i\l_i))\left(f(s^i,\l + \vec{1}\nabla_i \phi(s)) - f(s,\l) \right)\right.\\ \left. - \a \l_i \frac{\partial}{\partial \l_i}f(s,\l) + \frac{\s^2}{2} \frac{\partial^2}{\partial \l_i^2}f(s,\l)\right],
\end{multline}
where $\vec{1}$ is the $N$-dimensional vector with each component equal to 1 and $\nabla_i \phi(s) := \phi(s^i) - \phi(s)$.

For a clearer picture, consider the case in which $\a = \s = 0$. In this case, equation \eqref{7} is easily solved: $\l_i(t) = c - \b m_N(t)$, where $c = \l_0 + \b m_N(0)$. It follows that the process $s(t)$ is a Markov process with infinitesimal generator
\[
L^N f(s)  = \sum_{i=1}^N\left[ 1+ \tanh(s_i(c-\b m_N(s))) \right] [f(s^i) - f(s)].
\]
This is a modification of the standard stochastic dynamics of the Curie-Weiss model, where, in the usual version, $1+ \tanh(s_i(c-\b m_N(s)))$ is replaced by $\exp[s_i(c-\b m_N(s))]$, where $\b$ is the {\em inverse temperature} and $c$ is the {\em magnetic field}. As $\a >0$, in the dynamics of  $\l_i$ a dissipation, i.e. attraction to zero, is added to the jumps; as $\s>0$, a noise in this dynamics is introduced, in addition to the Poisson-type noise driving the spin-flips.

\begin{rem}
For the spin flip rates, that we set equal to $1+\tanh(s_i \l_i)$, other choices could be made, for instance $\exp[s_i \l_i]$, that would lead to similar results. The boundedness of our spin flip rates is convenient for the proofs, but we believe it is not an essential ingredient.
\end{rem}


\subsection{The macroscopic equation}

We now derive the dynamics of the system in the limit as $N \rightarrow +\infty$. In this section we proceed at a heuristic level; a formal proof will be sketched in the Appendix.

We introduce the following \emph{empirical measure}
\be[empirical_sum]
P^N:=\f1N\sum_{j=1}^N\delta_{(s_j,\l_j)}\ .
\ee
When $s = s(t)$, $\l = \l(t)$, we write $P^N_t$ for $P^N$.
For an arbitrary function $f:\ \{-1,1\}\times\mbb{R}\ra \mbb{R}$, empirical averages will be written in the form 
\[
 \bk{P^N_t,f} := \int f dP_t^N =\f1N\sum_j^N f(s_j(t),\l_j(t))\ .
\]
The infinitesimal generator \eqref{8} applied to a function $\Phi$ of the form
\be
\Phi(P^N):=\phi(\bk{P^N,f}) 
\ee
reads
\begin{multline}\label{generator}
\m{L}^N \phi(\bk{P^N,f}) = \sum_{j=1}^N \left[(1+\tanh(s_j\l_j))\left(\phi(\bk{P^{N,j},f})-\phi(\bk{P^N,f}) \right) \right.\\
\left.- \a \l_i \frac{\partial}{\partial \l_i}\phi(\bk{P^N,f}) + \frac{\s^2}{2} \frac{\partial^2}{\partial \l_i^2}\phi(\bk{P^N,f})\right]\ ,
\end{multline}
with
\[
\bk{P^{N,j},f}=\f1N\sum\bk{\delta_{(s_k^j,\l^j_k)},f}\ ,
\]
where
\be[jump]
s_k^j=s_k(1-\delta_{jk})-\delta_{jk}s_k\ ,\qquad \l_k^j=\l_k+\f{2\beta s_j}{N}\ 
\ee
are the spins and the $\l$'s after the jump of $s_j$.\\
If we expand the generator~\eqref{generator} for large $N$ we obtain
\begin{multline}
 \m{L}^N \phi(\bk{P^N,f})\approx \sum_{j=1}^N(1+\tanh(s_j\l_j))\phi'(\bk{P^N,f} )\\
\times\left(\f{2\b s_j}{N}\bk{P^N,\pp_\l f}+\f1N\bk{\de_{(-s_j,\l_j)}-\de_{(s_j,\l_j)},f}\right)\\
+\phi'(\bk{P^N,f})\left(-\a\bk{P^N,\l \pp_\l f}+\f{\s^2}{2}\bk{P^N,\pp_\l^2 f}\right)+O\left(\f1N\right)\ .
\end{multline}
This allows to identify the weak limit $P_t$ of the empirical measures $P^N_t$, evaluated along the paths of the process \eqref{8}, as the (deterministic) solution of the equation
\be[lim_dyn]
\bk{P_t,f}-\bk{P_0,f}=\int_0^t\bk{P_u,\m{L}(P_u)f}\d u\ ,
\ee
where
\begin{multline}\label{lim_gen}
\m{L}(P_t)f(s,\l)=\left((1+s\tanh(\l))(f(-s,\l)-f(s,\l))\right.\\
\left. +2\b\bk{P_t,s +\tanh(\l)}\pp_\l f(s,\l)-\a\l \pp_\l f(s,\l)+\f{\s^2}{2}\pp_\l^2 f(s,\l)\right)\ .
\end{multline}
We remark that the operator~\eqref{lim_gen} can be associated to the {\em nonlinear Markov Process} (see \cite{kolokoltsov})  $\left\{(\S_t,\L_t)\right\}_{t\geq 0}$, with $\S_t\in \{-1,+1\}$ and $\L_t\in\mathbb{R}$, solution of the stochastic equation
\be[proc]
\left\{
\begin{aligned}
&\S_t \rightarrow -\S_t\quad \text{with intensity}\quad 1+\tanh(\S_t\L_t)\ ,\\
&\d\L_t=\left(-\a\L_t+2\b\bk{P_t, s + \tanh(\l)}\right)\d t+\s\d B_t\ , \\
& P_t = \mbox{law}(\S_t,\L_t)\ .
\end{aligned}
\right.
\ee

Alternatively, writing formally $P_t = p_t(s,\l)\d\l$, $p_t$ can be identified as the weak solution of the nonlinear equation
\begin{multline}\label{FP}
\pp_t p_t(s,\l)=(1-s\tanh(\l))p_t(-s,\l)-(1+s\tanh(\l))p_t(s,\l)\\
+\f{\s^2}{2}\pp_\l^2 p_t(s,\l)-2\b g(t) \pp_\l p_t(s,\l)+\a\pp_\l (\l p_t(s,\l))\ ,
\end{multline}
where
\[
 g(t):=\sum_s\int_{-\infty}^{+\infty} p_t(s,\l)(s+\tanh(\l))\d\l\ .
\]
The regularizing effect of the second derivative guarantees that $p_t$ is indeed smooth in $\l$.

Notice that defining
\be[numu]
\nu_t(\l):=\sum_{s=\pm 1} p_t(s,\l)\ ,\qquad \mu_t(\l):=\sum_{s=\pm 1} sp_t(s,\l)\ ,
\ee
-- and so $p(+1,\l)=\f{\mu(\l)+\nu(\l)}{2}$, $p(-1,\l)=\f{-\mu(\l)+\nu(\l)}{2}$ --
you get from the Fokker-Planck equation~\eqref{FP}, the equivalent following system of PDEs
\be[FP2]
\left\{\begin{aligned}
 \pp_t\nu_t(\l)=&\f{\s^2}{2}\pp_\l^2 \nu_t(\l)+(\a\l-2\b g(t))\pp_\l \nu_t(\l)+\a\nu_t(\l)\ ,\\
\pp_t\mu_t(\l)=&\f{\s^2}{2}\pp_\l^2\mu_t(\l)+(\a\l-2\b g(t))\pp_\l \mu_t(\l)+(\a-2)\mu_t(\l)\\
&-2\tanh(\l)\nu_t(\l)\ ,
\end{aligned}\right.
\ee
where 
\be[g]
 g(t)=\bk{\nu_t,\tanh(\l)}+\bk{\mu_t,1}=\bk{\nu_t,\tanh(\l)}+m(t)\ ,
\ee
with $m(t)$ the mass of $\mu$ at time $t$ or, equivalently, the expected value of $\S_t$ (the `average spin' or magnetization).
Clearly, the knowledge of the pair $(\nu_t,\mu_t)$ of measure-valued processes, is equivalent to that of  $P_t$. By definition, $\nu_t(\l)$ is the density of the marginal distribution $\L_t$, while $\mu_t$ is the density of a signed measure; they must satisfy $\int_\mathbb{R}\nu_t(\l)\d\l=1$, $\int_\mathbb{R}\mu_t(\l)\d\l=m(t)\in [-1,1]$.

All the above can be translated into the following rigorous statement, whose proof is given in the Appendix. In Theorem~\ref{th:limit} below, we suppose that the laws of $(\sigma^{N}(0),\lambda^{N}(0))$ are $P_{0}$-chaotic. Recall the definition of chaoticity. Let $\theta$ be a probability measure on a Polish space $\mathcal{X}$ and, for $N\in \mathbb{N}$, let $\Theta_{N}$ be a symmetric probability measure on the product space $\mathcal{X}^{N}$ (the law of $(\sigma^{N}(0),\lambda^{N}(0))$ is a probability measure on $(\{-1,1\} \times \mathbb{R})^{N}$, assumed to be symmetric). Then $(\Theta_{N})_{N\in\mathbb{N}}$ is said to be $\theta$-chaotic if for every $n\in \mathbb{N}$ the joint law of the first $n$ marginals of $\Theta_{N}$ converges weakly to the product measure $\otimes^{n}\theta$.

\begin{theorem} \label{th:limit}
Assume that the initial conditions $(\s(0),\l(0)) = (\s^N(0),\l^N(0))$ for the processes \eqref{8} are $P_0$-chaotic for some probability measure $P_0$ on $\{-1,1\} \times \mathbb{R}$. Then the sequence of measure-valued random variables $(P^N)_{N\in \mathbb{N}}$ converges in distribution as $N \rightarrow +\infty$, in the topology of weak convergence of probability measures, to the law $P$ on path space of the unique solution of equation \eqref{proc} with initial distribution $P_{0}$; moreover, $(P_{t})_{t\geq 0}$, the measure-valued process of time marginals of $P$, solves the integral equation \eqref{lim_dyn}.
\end{theorem}

\section{The case without noise}

In this section we analyze equation \eqref{lim_dyn} in the special case of $\s = 0$. We also assume that the initial condition is such that all the spins have the same $\l$, i.e. we take $\l_j(0)=\l_0$ for all $j$'s. These simplifications allow a detailed analysis of the long-time behavior of the solution of \eqref{lim_dyn}.
Indeed, writing $P_0$ in the form
\be
P_0(s,\d\l)=\de (\l-\l_0)\times d_0(s)\ ,
\ee
where $d_0$ is a probability on $\{-1,1\}$, the corresponding solution of \eqref{lim_dyn} maintains the same form:
\be[ans]
P_t(s,\d\l)=\de(\l-\l(t))\times d_t(s)\ ,
\ee
Plugging \eqref{ans} into \eqref{lim_dyn}, and setting $m(t) := d_t(1) - d_t(-1) = \bk{d_t,s}$, 
we have that the measure \eqref{ans} is indeed a solution of the limiting dynamics \eqref{lim_dyn}, provided that
\be[dyn0]
\left\{\begin{aligned}
 \dot\l(t) &=2\b(m(t)+\tanh(\l(t)))-\a\l(t)\ ,\\
\dot m(t)&=-2(m(t)+\tanh(\l(t)))\ ;
\end{aligned}\right.
\ee
with
$\l(0)=\l_0$ and $m(0)=\bk{d_0,s}$.\\
Notice that, in terms of the pair $\mu_t,\nu_t$ introduced in \eqref{numu}, the solution \eqref{ans} reads
\be
\nu_t(\l)=\de(\l-\l(t))\ ,\quad \mu_t(\l)=m(t)\de(\l-\l(t))\ .
\ee

\subsection*{Analysis of the attractors}
First of all, we note that the only fixed point in \eqref{dyn0} is the origin $(0,0)$ of the phase plane $m,\l$. The linearization around this point gives
\be
\begin{pmatrix}
 \dot m\\
\dot\l
\end{pmatrix}
=
\begin{pmatrix}
 -2 & -2\\
  2\b & 2\b-\a
\end{pmatrix}
\begin{pmatrix}
 m\\
\l
\end{pmatrix}
\ee
and the eigenvalues of the system are
\be
x_\pm=\b-1-\f{\a}{2}\pm\sqrt{\left(\b-1-\f{\a}{2}\right)^2-2\a}\ .
\ee
These eigenvalues have both negative real part for $\b<\f{\a}{2}+1$, and both positive real part for $\b>\f{\a}{2}+1$. Thus, for $\b >\f{\a}{2}+1$, the local stability of the origin is lost.
Much more than local stability can be obtained for system \eqref{dyn0}.

\begin{theorem} \label{th:cycle}
\begin{itemize}
\item[(i)]
For $\b\leq\f{\a}{2}+1$ the origin is a global attractor for \eqref{dyn0}.
\item[(ii)] 
For $\b>\f{\a}{2}+1$ the system \eqref{dyn0} has a unique periodic orbit, which attracts all trajectories except the fixed point.
\end{itemize}
\end{theorem}
\begin{proof}
It is useful to perform a simple change of variable, consisting in replacing $m$ by $y:=2(\l+\b m)$. In the variables $(y,\l)$, system \eqref{dyn0} becomes
\be[lien]
\begin{aligned}
\dot y&=-2\a\l ,\\
\dot\l&=y-g(\l)\ ,
\end{aligned}
\ee
with $g(\l)=(2+\a)\l-2\b\tanh(\l)$. The system \eqref{lien} is of the {\em Liénard} type (see, for example, \cite{Carletti-Villari,Sabatini-Villari}), which allows a detailed study of the global stability. 

\noindent
{\em Case $\b\leq\f{\a}{2}+1$}. In this case it is easy to show that the function $g$ is strictly increasing, and it is odd. Setting
\be[W]
W(\l,y):=\f{\l^2}{2}+\f{y^2}{4\a}\ ,
\ee
we have	
\be[dotW]
\dot W=-\l g(\l)\ ,
\ee
where
\[
\dot W(\l,y) = \frac{d}{dt} W(\l(t), y(t))\Big|_{t=0}\ ,
\]
and $(\l(t), y(t))$ solves \eqref{lien} with $(\l(0),y(0)) = (\l,y)$.
Thus $W$ is a global {\em Lyapunov function}, which implies global stability of the origin.

\noindent
{\em Case $\b>\f{\a}{2}+1$}. The odd function $g$ has now two additional, symmetric zeros $\pm \l^*$, $\l^*>0$ (see figure~\ref{fig3}). Thus, the system \eqref{dyn0} satisfies the condition of Theorem 1.1 in \cite{Carletti-Villari}, which establishes existence and uniqueness of a globally stable periodic orbit.\\
\begin{figure}[htbp]
\begin{center}
\includegraphics{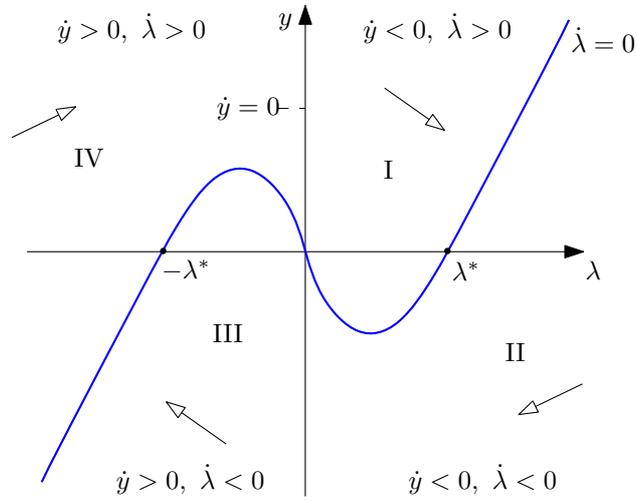}
\caption{Qualitative behavior of the dynamical system \eqref{lien}. The four regions indicate the four different directions that the vector field assumes. The blue curve is the graph of $g$.}
\label{fig3}
\end{center}
\end{figure}

Here we briefly sketch a proof of the existence, uniqueness and stability of the limit cycle.
First of all, let us prove that all the trajectories revolve around the origin. Suppose to start in region I (see figure~\ref{fig3}): $y$ is decreasing, $\l$ increasing, so you will eventually hit the nullcline $y=g$ (the origin is the only fixed point, and it is repelling). So we are in II: both $y$ and $\l$ are decreasing, so either you go in III or you `die' at infinity. Let us examine the case $y\rightarrow -\infty$; for the slope of the trajectory in this region we have
\be
\f{\d y}{\d \l}=\f{\dot y}{\dot\l}=\f{-2\a \l}{y-g(\l)}\xrightarrow[y\rightarrow -\infty]{}0
\ee  
which means that we will always end up in III. III and IV behave the same, by symmetry.

\noindent Denoting with $y_0$ ($y_1$) the intersection of the orbit with the positive (negative) $y$-axis (which we have proved to exist for every trajectory), the zeros of the function
\be
\Delta W(y_0):=W(0,y_1)-W(0,y_0)=\f{y^2_1-y_0^2}{4\a}\ ,
\ee
correspond to periodic orbits (we make reference to figure~\eqref{fig4} for notations).\\
\begin{figure}[htbp]
\begin{center}
\includegraphics{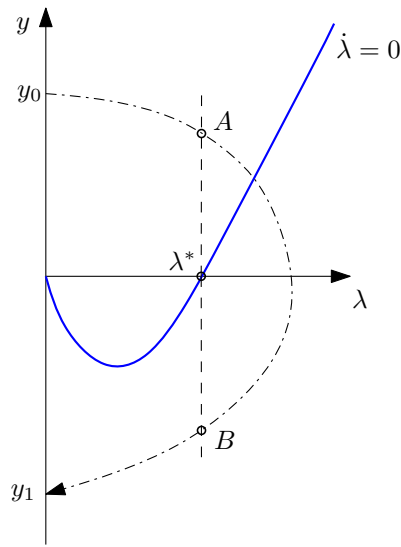}
\caption{Plot of the phase plane of the system \eqref{lien}. One (half) trajectory is sketched (dotted curve), starting at $(0,y_0)$  and then back to the $y$-axis at $(0,y_1)$. $A$ ($B$) denotes the point in which the $\lambda$-component of such trajectory becomes greater (lower) than $\l^*$.}
\label{fig4}
\end{center}
\end{figure}

Name $y^*_0$ the positive $y$-axis intersection of the orbit that passes through $(\l^*,0)$; then, calling $t_1$ the time in which $(0,y_1)$ is reached, you have
\be
\Delta W(y_0^*)=\int_0^{t_1}\dot W(t)\d t=-\int_0^{t_1}\l(t)g(\l(t))\d t\ >0\ ,
\ee
 since $g(\l)>0$ for $\l \in (0,\l^*)$. For $y_0<y_0^*$ you get an $x$-axis intersection smaller than $\l^*$ --for the uniqueness of trajectories. Then you still have $\Delta W>0$.\\
Now take $y_0>y_0^*$. In this case it is convenient to split $\Delta W$ as follows
\be[split]
\Delta W(y_0)=\underset{(a)}{-\int_0^{t_B} \l(t)g(\l(t))\d t}\underset{(b)}{-\int_{t_A}^{t_B} \l(t)g(\l(t))\d t}\underset{(c)}{-\int_{t_B}^{t_1} \l(t)g(\l(t))\d t}\ .
\ee
The first term is clearly positive, and with a change of variable it can be rewritten as
\be
(a)=\int_0^{\l^*}\f{\l\d\l}{1-y(\l)/g(\l)}\ .
\ee
We claim that 
\be[claim]
(a)\searrow 0\quad \text{monotonically as}\quad y_0\ra +\infty\ .
\ee
In order to prove this claim, first of all notice that if one starts at $(0,\wt{y}_0=y_0+\epsilon)$ for $\epsilon$ positive and arbitrarily small, then $\wt{y}(\l)>y(\l)$ for every $\l$ between zero and the $\l$-axis intersection with the trajectory. This  implies that $(a)$ decreases monotonically with $y_0$. Moreover, the slope of the trajectories
\[
\f{\d y}{\d\l}=\f{2\a\l}{g(\l)-y}
\]
 is bounded as long as we are away from the nullcline $y=g(\l)$ and on a compact interval in $\l$. So, given an arbitrary positive number $M$ you can always find an initial condition $(0,y_0)$  and a $\wt{\l}\geq \l^*$ such that $y(\l)>M$  as long as $\l\in (0,\wt{\l})$. This proves the claim.\\
Analogously, one can prove that $(c)$ is positive and monotonically decreasing to zero as $y_0\ra +\infty$.

\noindent Now let us deal with the second term in \eqref{split}. It can be rewritten as 
\be[(b)]
(b)=-\int_{y(t_{B})}^{y(t_A)}g(\l(y))\d y\ ,
\ee
which is negative. We claim that $(b)\searrow -\infty$ monotonically as $y_0\ra +\infty$. In order to show this, it is convenient to split $(b)$  as follows
\be[bsplit]
(b)=\int_{t_A}^{t_A+\de}\dot W\d t+\int_{t_A+\de}^{t_{B}-\de}\dot W\d t+\int_{t_{B}-\de}^{t_{B}}\dot W\d t\ ,
\ee
with $\de$ arbitrarily small (and positive). The first and third terms in the right hand side of \eqref{bsplit} remain negative and finite in the limit $y_0\ra +\infty$. For the second one, namely 
\be
-\int^{y(t_{A}+\de)}_{y(t_{B}-\de)}g(\l(y))\d y\ ,
\ee
the integrand is bounded away from zero (which is not true for \eqref{(b)}) and $y(t_A+\de)\nearrow +\infty$ monotonically when $y_0\ra +\infty$, which comes as a consequence of the proof of \eqref{claim}. This is enough to prove the claim.

\noindent In the end, we have that
\begin{itemize}
\item $\Delta W(y_0)$ is positive when $y_0\leq y_0^*$
\item $\Delta W(y_0)\searrow -\infty$ monotonically as $y_0\ra +\infty$
\end{itemize}
which prove the existence and uniqueness of the periodic orbit. In order to prove  stability, it is enough to say that $\Delta W>0$ when $y_0<y_0^p$ and $\Delta W<0$ when  $y_0>y_0^p$, where $y_0^p$ denotes the positive $y$-axis intersection of the periodic orbit.
\end{proof}

Summarizing, the simple model presented in this section is obtained from a standard Curie-Weiss-type model by introducing a dissipation on the spin-flip intensity. This dissipation does not destroy self-organization of the spins. However, the nonzero magnetization produced by the self-organization does not converge to a constant value, as in Glauber dynamics for ferromagnets, but rather oscillates periodically.

\section{Numerical results and conclusions}

Here we would like to present and briefly discuss some numerical results regarding the system of partial differential equations~\eqref{FP2}, that is the general case with the extra noise in the dynamics of the intensities $\l_j$'s~\eqref{7}, whose qualitative analysis is intended to be subject of future work.
The qualitative idea that we would like to stress (supported by the numerical tests) is the following: for small enough values of the diffusion parameter $\s$ the system qualitatively behaves just as in the no-diffusion case we have analyzed in the previous section, i.e. the extra noise gives just small perturbations around the periodic orbit (for the super-critical case $\b>\a/2 +1$) or around the totally disordered configuration (when $\b<\a/2 +1$); for $\s$ large enough, instead, the diffusion term dominates and  both the super-critical and sub-critical cases evolve to gaussian behaviors around the totally disordered configuration $(m,\l)=0$.

Figure~\ref{fig5} refers to the sub-critical case, with parameters $\a=3$, $\b=1$. 
There it is shown a comparison between two phase plots: one is the phase plot of $m = \bk{\mu,1}$ versus the expected value of $\l$ 
\[
 \bk{\l}_t:=\int_\mbb{R}x\nu_t(\d x)\ ,
\]
and with $\s = 0.1$,
the other one is the phase plot of the corresponding $\s=0$ case, i.e. it is calculated by solving system~\eqref{dyn0}. Here it is clear that the trigger of a small (but non-zero) value of extra noise does not spoil the qualitative behavior of the system.
Moreover, from equation \eqref{proc}, it is easy to see that the variance of $\nu_t(\l)$
\[
\bk{(\l - \bk{\l}_t)^2}_t = \Var(\L_t)
\]
satisfies the equation
\[
\frac{\d}{\d t}\Var(\L_t) = - 2 \a \Var(\L_t) + \s^2,
\]
so that
\be[varnu]
\Var(\L_t) = e^{-2 \a t} \Var(\L_0) + \frac{\s^2}{2\a}\left(1-e^{-2\a t}\right).
\ee
In all simulations the initial condition for $\nu$ is (approximately) a delta function at $\l_0=3$, so $\Var(\L_0) \simeq 0$. For large $t$, the variance approaches the value $\frac{\s^2}{2\a}$. 


An analogous situation is found in the super-critical case, as can be seen in figure~\ref{fig7}
One can see that the mean of the densities keeps oscillating in time and  that the periodic behavior is indeed preserved, even if slightly modified, in the small noise case. By \eqref{varnu}, the variance of $\nu$ remains bounded and actually quite small during the evolution, thus showing that turning on a small noise does not spoil the qualitative behavior of the noiseless system.

Figures~\ref{fig9} and 
\ref{fig11}, 
are the analogues of the preceding ones, but when $\s=10$, i.e. with a large contribution of the extra-noise in the dynamics of the intensities~\eqref{7}: it is clear that the diffusion is here predominant, pushing the variance up to a gaussian behavior around the totally disordered configuration $(m,\l)=(0,0)$. The more interesting result is perhaps the one shown in figure~\ref{fig11}: the periodic behavior of the no-diffusion case is completely lost and the system rapidly evolves to the origin of the phase plane.

\begin{figure}[p]
\begin{minipage}{.44\textwidth}
 \includegraphics[width=\textwidth]{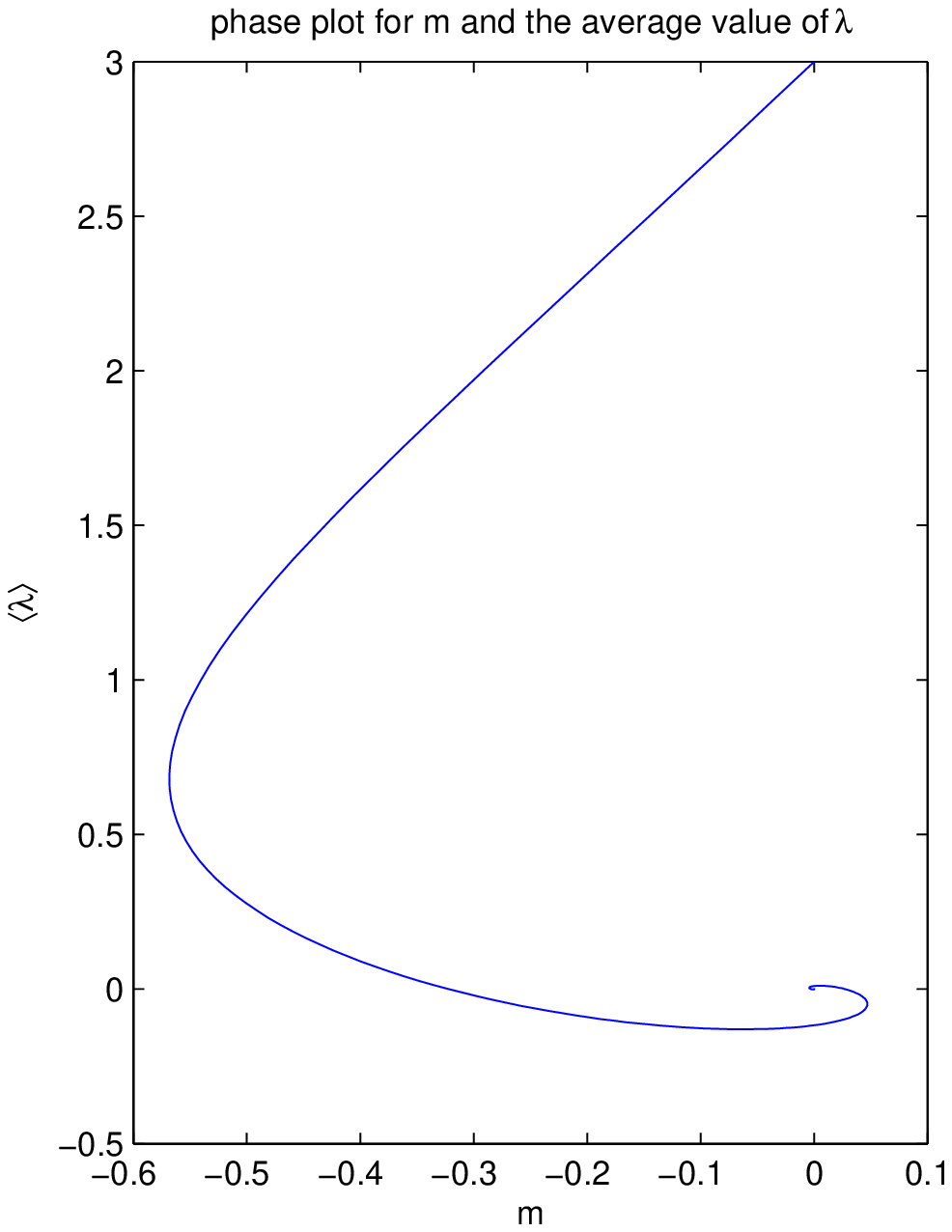}
\end{minipage}%
\begin{minipage}{.44\textwidth}
 \includegraphics[width=\textwidth]{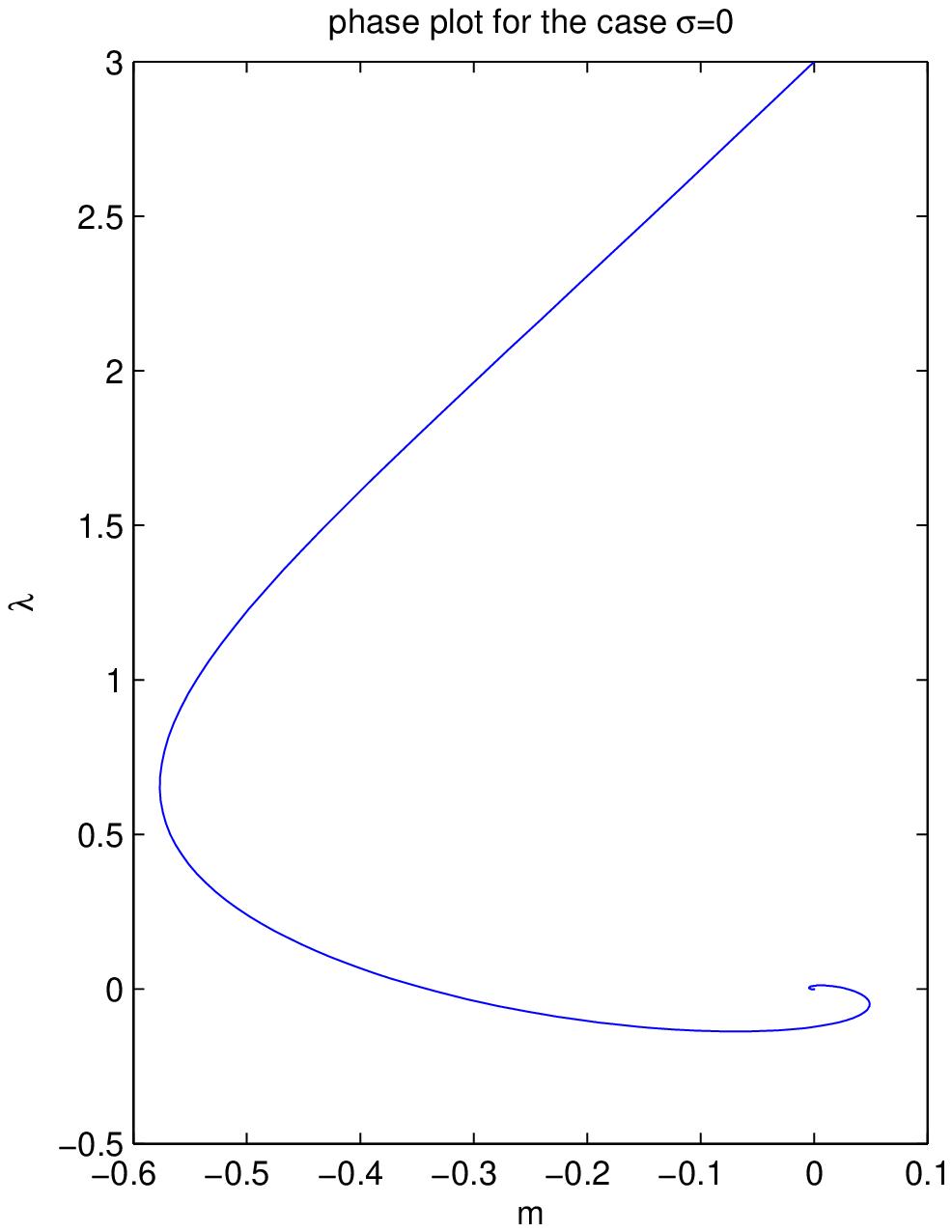}
\end{minipage}%
\caption{Phase plot of the expected value of $\l$ versus $m$ (left) and phase plot of the no-diffusion case (namely of the dynamical system \eqref{dyn0}) (right). The parameters are   $\a=3$, $\b=1$, $\s=0.1$, $\l_0=3$, $m(0)=0$.}  \label{fig5}
\end{figure}

%

\begin{figure}[p]
\begin{minipage}{.44\textwidth}
 \includegraphics[width=\textwidth]{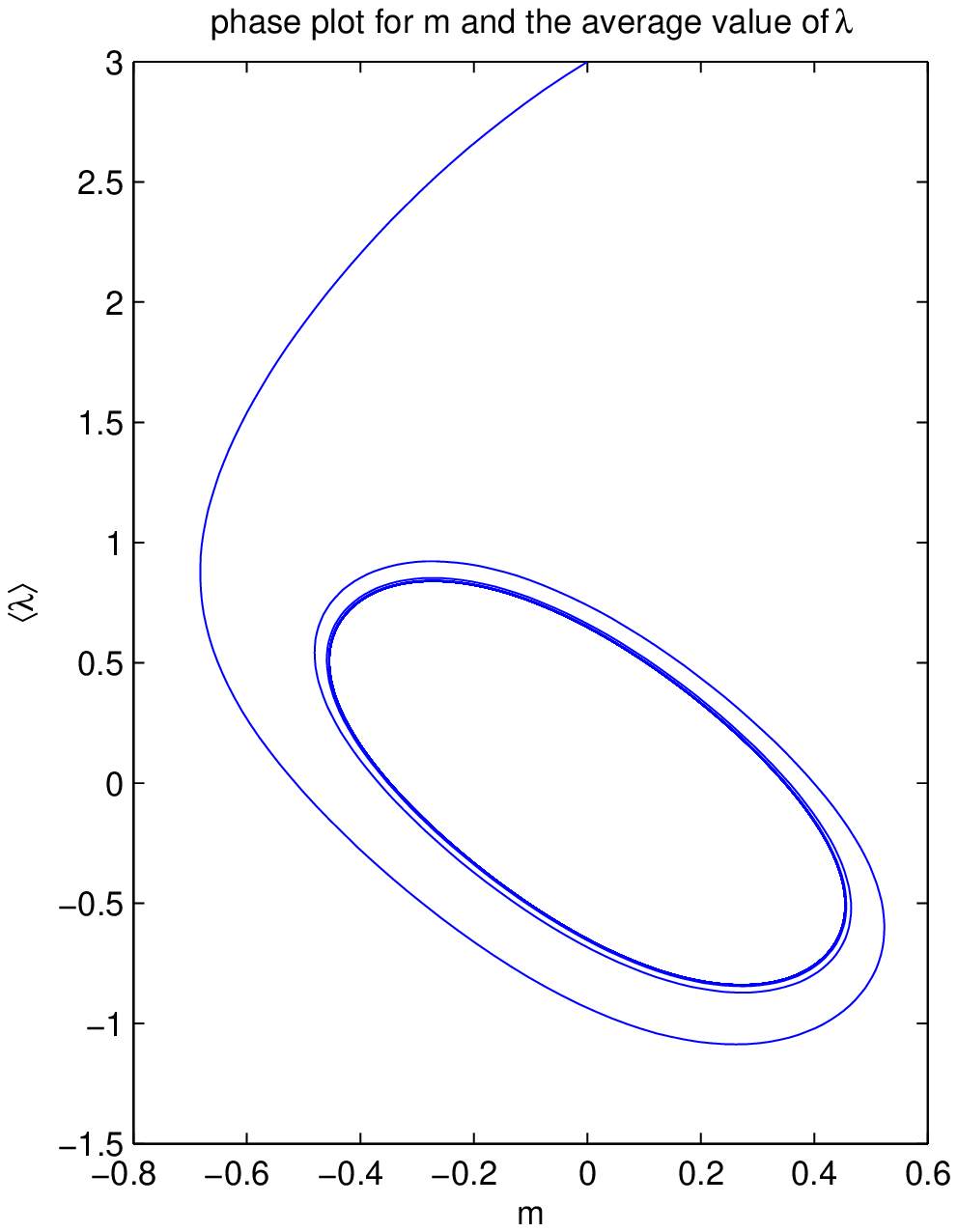}
\end{minipage}%
\begin{minipage}{.44\textwidth}
 \includegraphics[width=\textwidth]{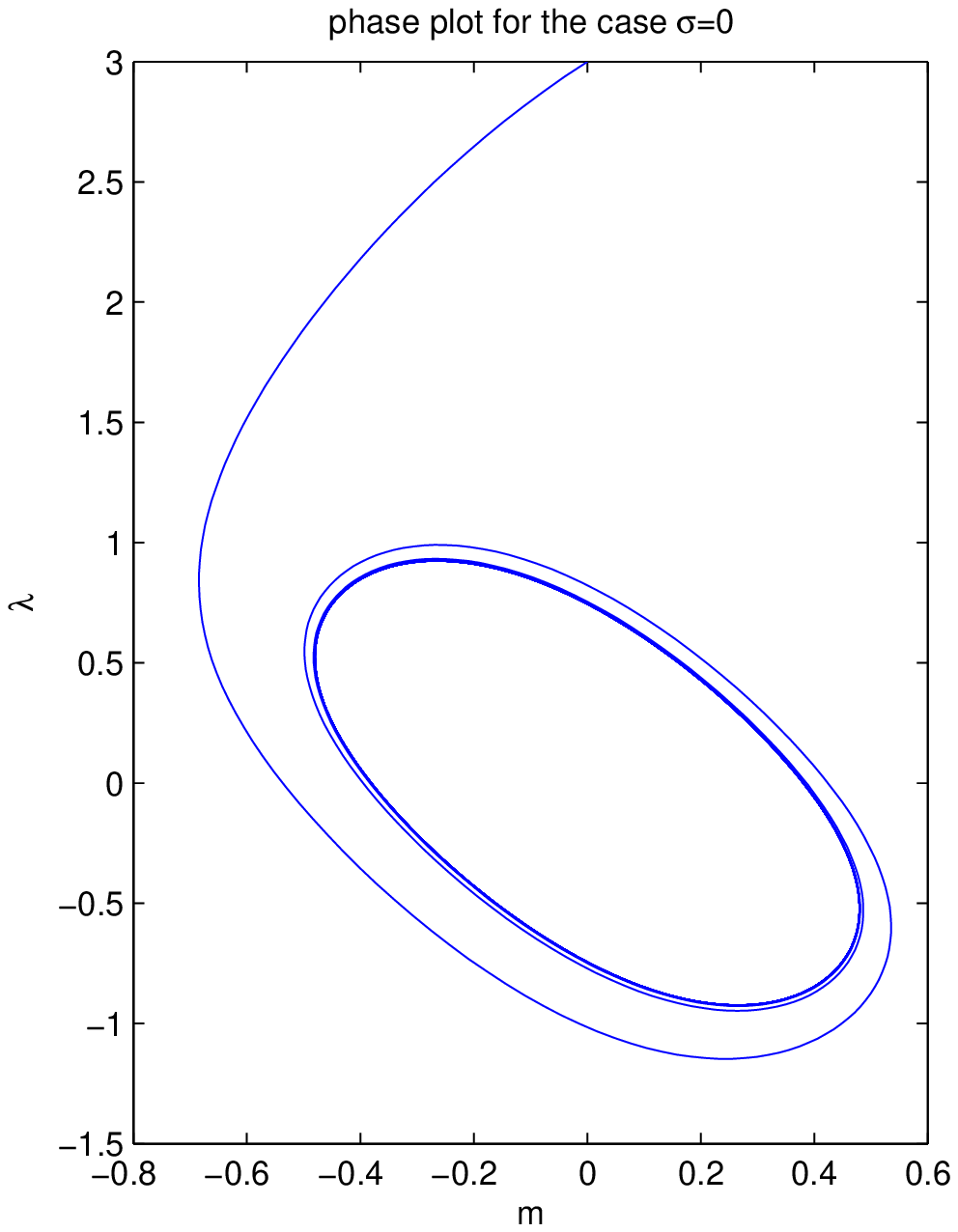}
\end{minipage}%
\caption{Phase plot of the expected value of $\l$ versus $m$ (left) and phase plot of the no-diffusion case (right). The parameters are   $\a=3$, $\b=3$, $\s=0.1$, $\l_0=3$ and $m(0)=0$.}  \label{fig7}
\end{figure}

%

\newpage

\begin{figure}[p]
\begin{minipage}{.45\textwidth}
 \includegraphics[width=\textwidth]{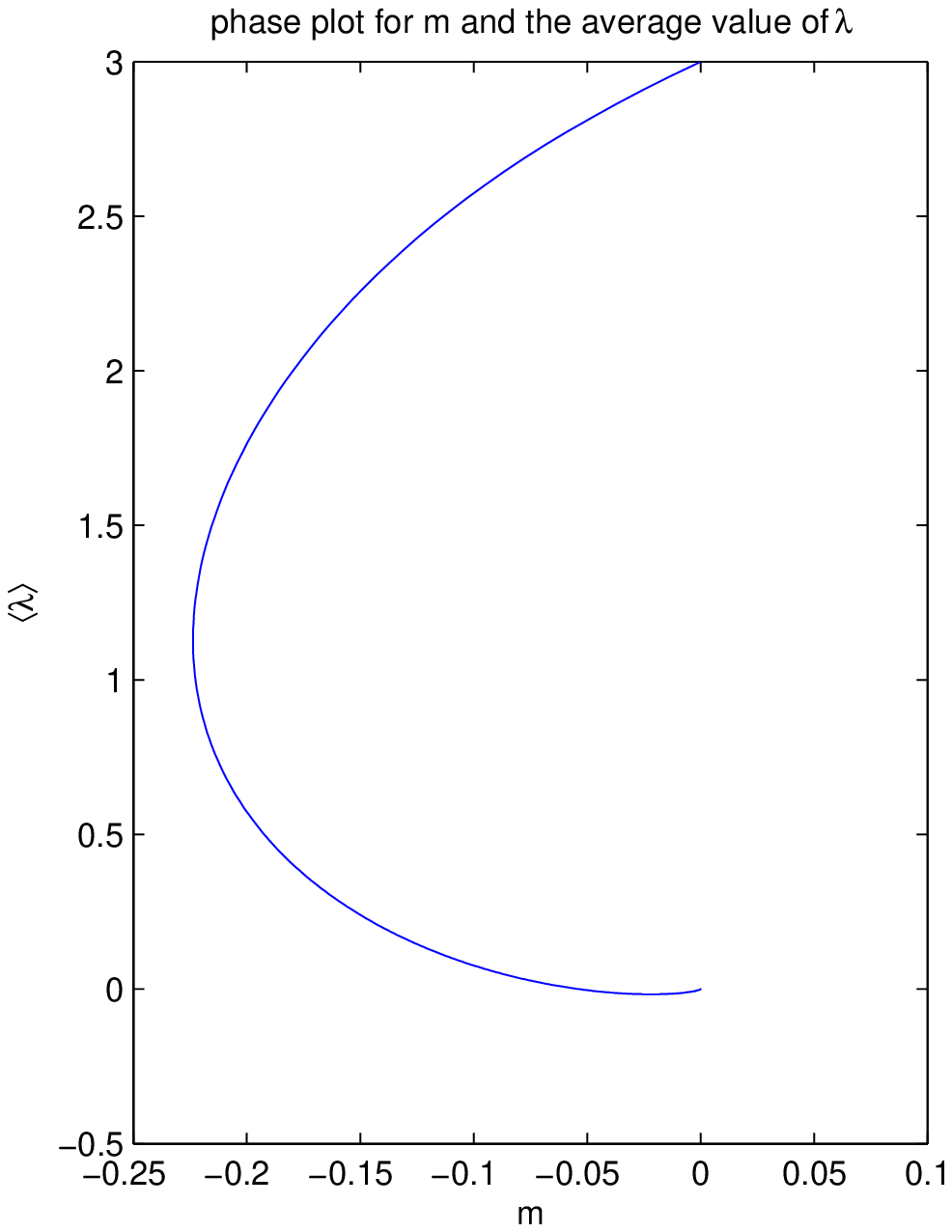}
\end{minipage}%
\begin{minipage}{.45\textwidth}
 \includegraphics[width=\textwidth]{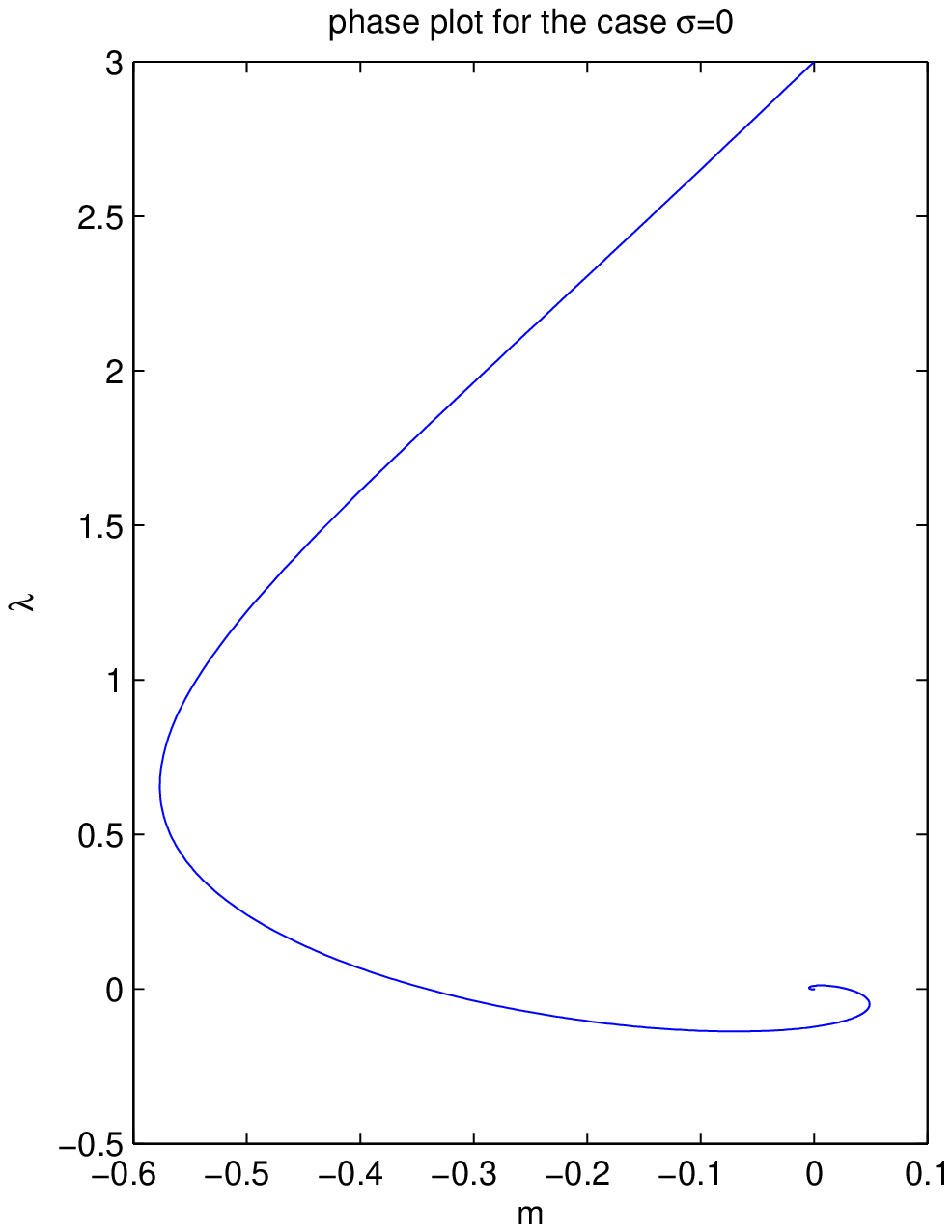}
\end{minipage}%
\caption{Phase plot of the expected value of $\l$ versus $m$ (left) and phase plot of the no-diffusion case (right). The parameters are  $\a=3$, $\b=1$, $\s=10$, $\l_0=3$ and $m(0)=0$.} \label{fig9} 
\end{figure}


\begin{figure}[p]
\begin{minipage}{.45\textwidth}
 \includegraphics[width=\textwidth]{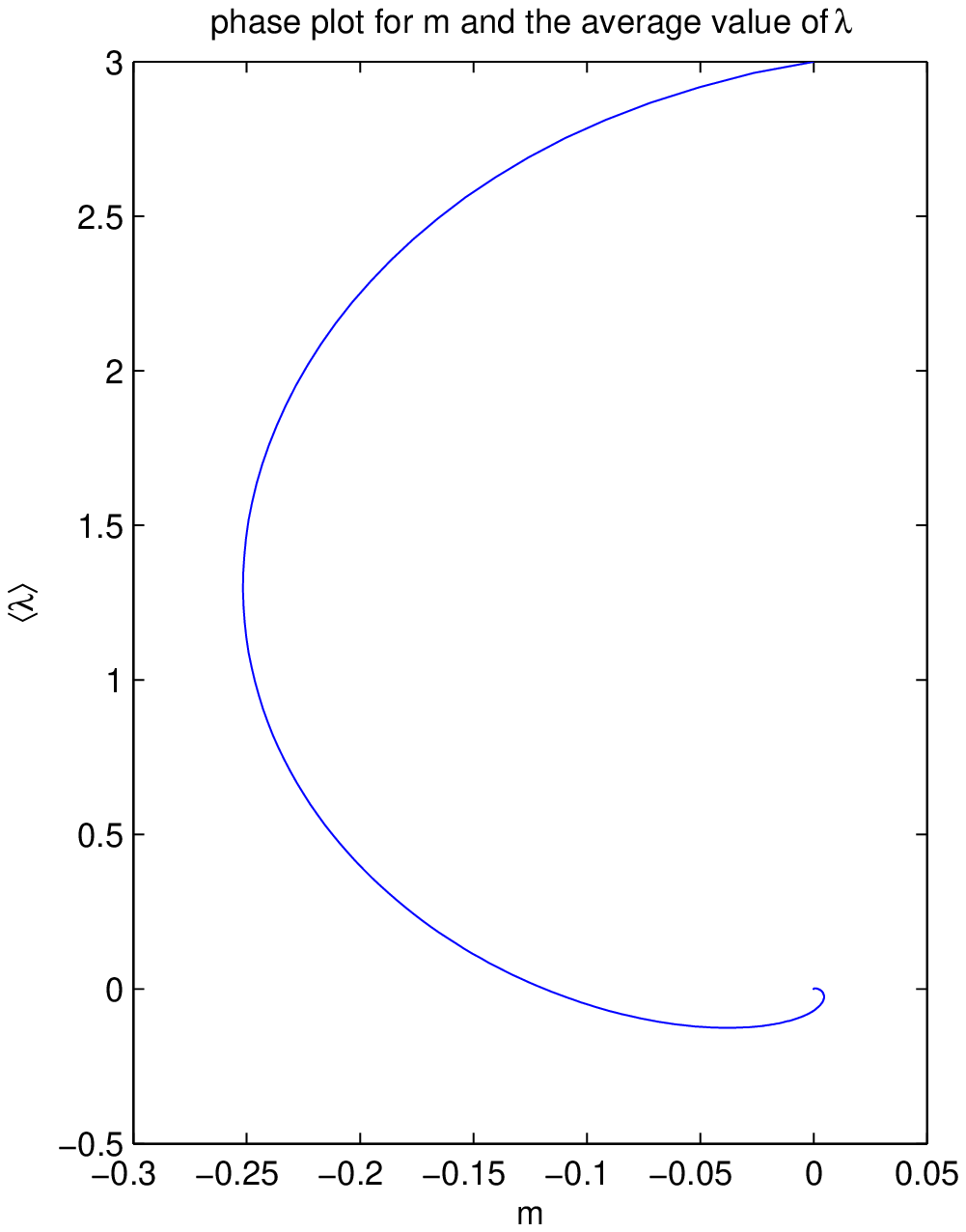}
\end{minipage}%
\begin{minipage}{.45\textwidth}
 \includegraphics[width=\textwidth]{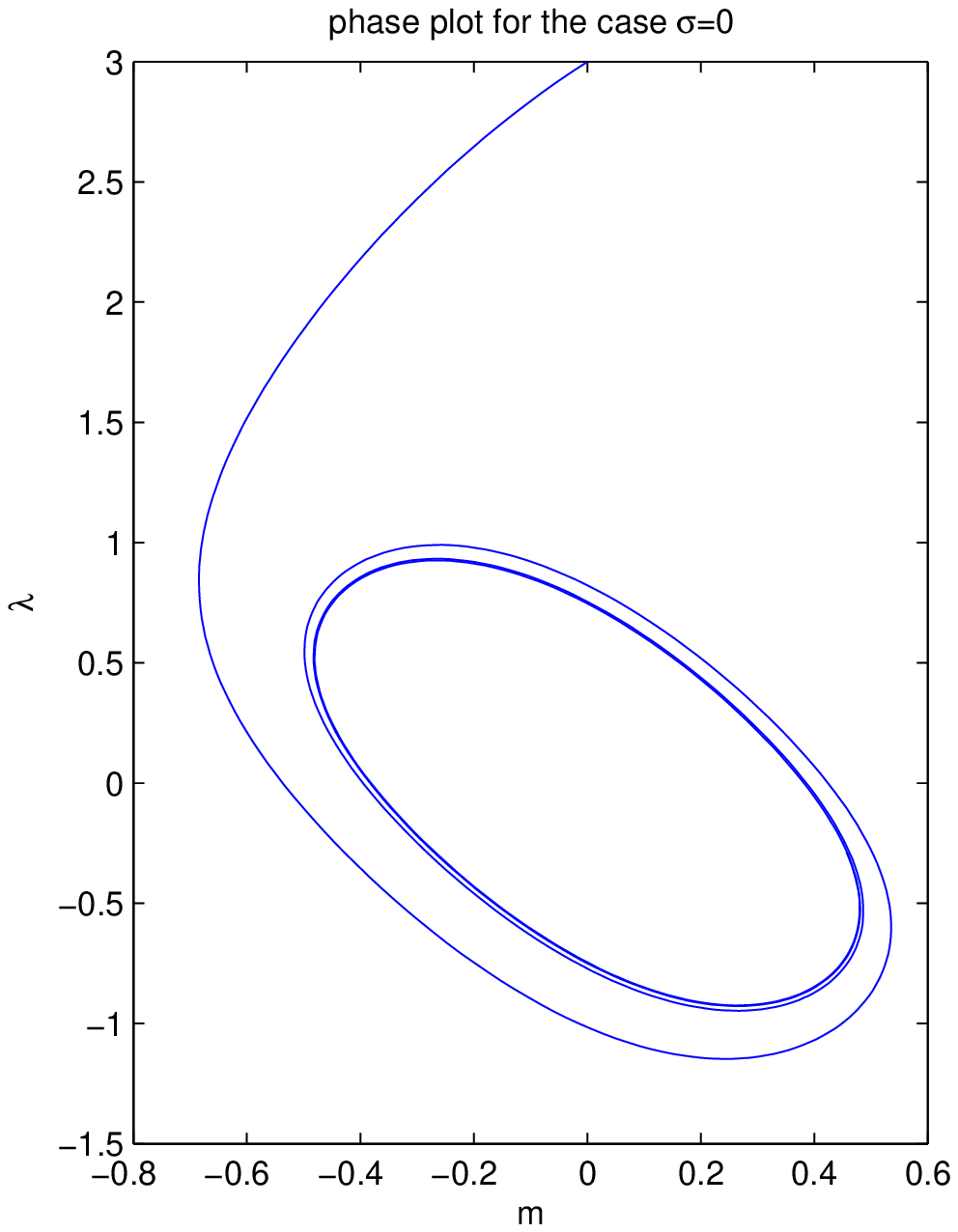}
\end{minipage}%
\caption{Phase plot of the expected value of $\l$ versus $m$ (left) and phase plot of the no-diffusion case (right). The parameters are    $\a=3$, $\b=3$, $\s=10$ and $\l_0=3$, $m(0)=0$.}\label{fig11}  
\end{figure}


\FloatBarrier

\begin{appendix}

\section{Proof of Theorem~\mbox{\ref{th:limit}}}

For $\mathcal{X}$ a Polish space, denote by $\mathcal{P}(\mathcal{X})$ the space of probability measures on the Borel sets of $\mathcal{X}$; equip $\mathcal{P}(\mathcal{X})$ with the topology of weak convergence, which makes it a Polish space, too. As in Section~\ref{sec:micromacro}, let $P^{N}$ be the empirical measure of the $N$-particle system. We may assume that the $\{-1,1\}\times\mathbb{R}$-valued processes $(s_{j},\lambda_{j})$ have c{\`a}dl{\`a}g trajectories (i.e., trajectories that are right-continuous with limits from the left). Consequently, $P^{N}$ is a probability measure on the Borel sets of $D:= D([0,\infty),\{-1,1\}\times\mathbb{R})$, the space of $\{-1,1\}\times\mathbb{R}$-valued c{\`a}dl{\`a}g functions equipped with the Skorohod topology.

The strategy of proof, here, is to represent both the microscopic and the macroscopic model as solutions of certain stochastic differential equations in order to apply results by \cite{graham1992} on propagation of chaos, which implies convergence of empirical measures. 

Let $\eta$ be Lebesgue measure restricted to the Borel sets on the interval $(0,2)$. Let $((\Omega,\mathcal{F},\Prb),(\mathcal{F}_{t})_{t\geq 0})$ be a filtered probability space satisfying the usual hypotheses rich enough to carry an independent family $(B_{i},\m{N}_{i})_{i\in \mathbb{N}}$ of one-dimensional $(\mathcal{F}_{t})$-Brownian motions $B_{i}$ and stationary $(\mathcal{F}_{t})$-Poisson random measures $\m{N}_{i}$ with characteristic measure $\eta$. For $N\in \mathbb{N}$, consider the system of It{\^o}-Skorohod equations
\begin{align} \label{eq:prelimit}
\begin{split}
	\d\lambda^{N}_{i}(t) &= -\alpha \lambda^{N}_{i}(t-)\d t + \sigma \d B_{i}(t) -\frac{\beta}{N}\sum_{k=1}^{N} \int_{(0,2)} q\left((s^{N}_{k}(t-),\lambda^{N}_{k}(t-)),u\right) \m{N}_{k}(\d u,\d t),\\
	\d s^{N}_{i}(t) &= \int_{(0,2)} q\left((s^{N}_{i}(t-),\lambda^{N}_{i}(t-)),u\right) \m{N}_{i}(\d u,\d t),\quad i\in \{1,\ldots,N\},
\end{split}
\end{align}
where
\[
	q((s,\lambda),u):= -2h(s)\cdot \mathbf{1}_{(0,1+h(s)\tanh(\lambda))}(u),\quad (s,\lambda)\in \mathbb{R}^{2},\; u\in (0,2),
\]
and $h(s):= (-1)\vee (s\wedge 1)$, $s\in \mathbb{R}$. By Theorem~1.2 in \cite{graham1992}, existence and uniqueness of solutions hold in the strong sense for the system of equations \eqref{eq:prelimit} since its coefficients are globally Lipschitz continuous; the jump coefficient, in particular, satisfies the $L^1$ Lipschitz assumption of the theorem.
Clearly, if $s\in \{-1,1\}$, then $h(s) = s$, $h(s)\tanh(\lambda) = \tanh(s\cdot\lambda)$, and
\[
	\int_{(0,2)} q((s,\lambda),u) \eta(\d u) = -2\left(s+\tanh(\lambda)\right).
\]
Thanks to the choice of the jump heights, if $(s^{N}(0),\lambda^{N}(0))$ is such that $s^{N}_{i}(0)\in \{-1,1\}$, then $s^{N}_{i}(t)\in \{-1,1\}$ for all $t\geq 0$. Let us fix a sequence of initial conditions $(s^{N}(0),\lambda^{N}(0))_{N\in \mathbb{N}}$ such that $s^{N}(0)\in \{-1,1\}^{N}$ for all $N\in \mathbb{N}$ and $(\Prb\circ(s^{N}(0),\lambda^{N}(0))^{-1})_{N\in \mathbb{N}}$ is $\mu$-chaotic for some $\mu\in \mathcal{P}(\{-1,1\}\times\mathbb{R})$. The solution process $(s^{N},\lambda^{N})$ is then a $\{-1,1\}^{N}\times\mathbb{R}^{N}$-valued Markov process. Comparing its infinitesimal generator (cf.\ equation~(1.2) in \cite{graham1992}) with equation~\eqref{8} above shows that $(s^{N},\lambda^{N})$ is a realization of the $N$-particle microscopic model. To prove Theorem~\ref{th:limit} we thus have to prove convergence of the empirical measures $P^{N} := \sum_{i=1}^{N} \delta_{(s^{N}_{i},\lambda^{N}_{i})}$ associated with the solutions of \eqref{eq:prelimit}.

Define a function $\bar{b}\!: \mathcal{P}(\mathbb{R}^{2}) \rightarrow \mathbb{R}$ by
\[
	\bar{b}(\nu):= 2\beta\cdot \int_{\mathbb{R}^{2}} \left(h(s) + \tanh(\lambda)\right)\nu(\d s, \d\lambda).
\]
Notice that $\bar{b}$ is Lipschitz continuous with respect to the bounded Lipschitz (or Dudley) metric on $\mathcal{P}(\mathbb{R}^2)$ as well as with respect to the Wasserstein-1 (or Lipschitz) metric on $\mathcal{P}_{1}(\mathbb{R}^2)$, the space of probability measures with finite first moments. By Theorem~2.1 in \cite{graham1992}, existence and uniqueness of solutions hold in the strong sense for the McKean-Vlasov It{\^o}-Skorohod equation
\begin{align} \label{eq:limit}
\begin{split}
	\d\Lambda(t) &= -\alpha \Lambda(t)dt + \sigma \d B_{1}(t) + \bar{b}(P_{t})\d t, \\
	\d\Sigma(t) &= \int_{(0,2)} q\left((\Sigma(t-),\Lambda(t)),u\right) \m{N}_{1}(\d u,\d t),\\
	P_{t} &= \mbox{law}(\Sigma(t),\Lambda(t)).
\end{split}
\end{align}
Assume that $P_{0} = \mu = \mbox{law}(\Sigma(0),\Lambda(0))$. Set $P:= \mbox{law}(\Sigma,\Lambda)$ and observe that $P \in \mathcal{P}(D)$. Comparison of infinitesimal generators yields that the solution $(\Sigma,\Lambda)$ of \eqref{eq:limit} is a realization of the nonlinear Markov process given by equation \eqref{proc}. Moreover, the $\mathcal{P}(\{-1,1\}\times\mathbb{R})$-valued process $(P_{t})_{t\geq 0}$ coincides with the solution of equation \eqref{lim_dyn} with initial condition $P_{0}$.

For $N\in \mathbb{N}$, consider the system of It{\^o}-Skorohod equations
\begin{align} \label{eq:prelimit2}
\begin{split}
	\d\bar{\lambda}^{N}_{i}(t) &= -\alpha \bar{\lambda}^{N}_{i}(t)\d t + \sigma \d B_{i}(t) + \bar{b}(\bar{P}^{N}_{t})\d t,\\
	\d\bar{s}^{N}_{i}(t) &= \int_{(0,2)} q\left((\bar{s}^{N}_{i}(t-),\bar{\lambda}^{N}_{i}(t)),u\right) \m{N}_{i}(\d u,\d t),\quad i\in \{1,\ldots,N\},
\end{split}
\end{align}
where $\bar{P}^{N}_{t}:= \sum_{i=1}^{N}\delta_{\bar{(\lambda}^{N}_{i}(t),\bar{s}^{N}_{i}(t))}$ is the empirical measure of the solution at time $t$. Notice that $\bar{\lambda}^{N}_{i}(t) = \bar{\lambda}^{N}_{i}(t-)$ by continuity of trajectories and that all processes are stochastically continuous. Again by Theorem~1.2 in \cite{graham1992}, existence and uniqueness of solutions hold in the strong sense for the system of equations \eqref{eq:prelimit2}. If the initial condition $(\bar{s}^{N}(0),\bar{\lambda}^{N}(0))$ for \eqref{eq:prelimit2} is such that $\bar{s}^{N}_{i}(0) \in \{-1,1\}$, then $\bar{s}^{N}_{i}(t)\in \{-1,1\}$ for all $t\geq 0$. Fix the initial condition at $(\bar{s}^{N}(0),\bar{\lambda}^{N}(0)):= (s^{N}(0),\lambda^{N}(0))$. Since $s^{N}(0)$ takes values in $\{-1,1\}^{N}$ and by the continuity of Lebesgue integrals, the system of equations \eqref{eq:prelimit2} can be rewritten as
\begin{align} \label{eq:prelimit2b}
\begin{split}
	\d\bar{\lambda}^{N}_{i}(t) &= -\alpha \bar{\lambda}^{N}_{i}(t-)\d t + \sigma \d B_{i}(t) -\frac{\beta}{N}\sum_{k=1}^{N} \int_{(0,2)} q\left((\bar{s}^{N}_{k}(t-),\bar{\lambda}^{N}_{k}(t-)),u\right)\eta(\d u)\d t,\\
	\d\bar{s}^{N}_{i}(t) &= \int_{(0,2)} q\left((\bar{s}^{N}_{i}(t-),\bar{\lambda}^{N}_{i}(t-)),u\right) \m{N}_{i}(\d u,\d t),\quad i\in \{1,\ldots,N\}.
\end{split}
\end{align}
Set $\bar{P}^{N}\doteq \sum_{i=1}^{N} \delta_{(\bar{s}^{N},\bar{\lambda}^{N})}$. By Theorem~4.1 in \cite{graham1992}, the sequence $(\mbox{law}(\bar{s}^{N},\bar{\lambda}^{N}))_{N\in\mathbb{N}}$ is $P$-chaotic. This implies, by the Tanaka-Sznitman theorem (for instance, Theorem~3.2 in \cite{gottlieb1998}), that the sequence $(\bar{P}^{N})_{N\in\mathbb{N}}$ of $\mathcal{P}(D)$-valued random variables converges in distribution to the probability measure $P$. In order to establish convergence of $(P^{N})_{N\in\mathbb{N}}$ to $P$, it is therefore enough to show that
\[
	\hat{d}_{bL}\left(\mbox{law}(P^{N}),\mbox{law}(\bar{P}^{N})\right) \stackrel{N\to\infty}{\longrightarrow} 0,
\]
where $\hat{d}_{bL}$ is the bounded Lipschitz metric on $\mathcal{P}(\mathcal{P}(D))$. By definition of $\hat{d}_{bL}$ and since both $P^{N}$ and $\bar{P}^{N}$ are empirical measures for processes defined on the same stochastic basis, we have
\[
	\hat{d}_{bL}\left(\mbox{law}(P^{N}),\mbox{law}(\bar{P}^{N})\right) \leq \Mean\left[ d_{bL}\left(P^{N},\bar{P}^{N}\right) \right] \leq \frac{1}{N}\sum_{i=1}^{N} \Mean\left[ d_{Sko}\left( (s^{N}_{i},\lambda^{N}_{i}), (\bar{s}^{N}_{i},\bar{\lambda}^{N}_{i}) \right) \right],
\]
where $d_{bL}$ is the bounded Lipschitz metric on $\mathcal{P}(D)$ and $d_{Sko}$ the Skorohod metric on $D$. For $i\in \mathbb{N}$, let $\tilde{\mathcal{N}}_{i}$ be the compensated Poisson random measure associated with $\mathcal{N}_{i}$, that is, $\tilde{\mathcal{N}}_{i}(\d u,\d t) = \mathcal{N}_{i}(\d u,\d t) - \eta(\d u)dt$. Then for $T > 0$, $i\in \{1,\ldots,N\}$, $N\in \mathbb{N}$,
\begin{align*}
	& \Mean\left[\sup_{t\in [0,T]} |\lambda^{N}_{i}(t) - \bar{\lambda}^{N}_{i}(t)|\right] \\
\begin{split}
	&\leq \alpha \Mean\left[ \int_{0}^{T} |\lambda^{N}_{i}(t-) - \bar{\lambda}^{N}_{i}(t-)|\d t \right] + \Mean\left[\sup_{t\in[0,T]} \left|\frac{\beta}{N} \sum_{k=1}^{N} \int_{0}^{t}\int_{(0,2)} q\left((s^{N}_{k}(r-),\lambda^{N}_{k}(r-)),u\right) \tilde{\m{N}}_{k}(\d u,\d r) \right|\right] \\
	&\quad + \frac{\beta}{N} \sum_{k=1}^{N} \Mean\left[\sup_{t\in[0,T]}\left| \int_{0}^{t}\int_{(0,2)} \left(q\left((s^{N}_{k}(r-),\lambda^{N}_{k}(r-)),u\right) - q\left((\bar{s}^{N}_{k}(r-),\bar{\lambda}^{N}_{k}(r-)),u\right)\right) \eta(\d u)\d r \right| \right]
\end{split}\\
\begin{split}
	&\leq \alpha \int_{0}^{T} \Mean\left[|\lambda^{N}_{i}(t-) - \bar{\lambda}^{N}_{i}(t-)|\right] \d t + 4\Mean\left[\left|\frac{\beta}{N} \sum_{k=1}^{N} \int_{0}^{T}\int_{(0,2)} q\left((s^{N}_{k}(t-),\lambda^{N}_{k}(t-)),u\right) \tilde{\m{N}}_{k}(\d u,\d t) \right|^{2} \right]^{1/2} \\
	&\quad + \frac{\beta}{N} \sum_{k=1}^{N} \Mean\left[\int_{0}^{T}\int_{(0,2)} \left|q\left((s^{N}_{k}(t-),\lambda^{N}_{k}(t-)),u\right) - q\left((\bar{s}^{N}_{k}(t-),\bar{\lambda}^{N}_{k}(t-)),u\right)\right| \d u\,\d t \right]
\end{split}\\
\begin{split}
	&\leq \alpha \int_{0}^{T} \Mean\left[|\lambda^{N}_{i}(t-) - \bar{\lambda}^{N}_{i}(t-)|\right] \d t + \frac{4\beta}{N}\Mean\left[\sum_{k=1}^{N} \int_{0}^{T}\int_{(0,2)} \left(q\left((s^{N}_{k}(t-),\lambda^{N}_{k}(t-)),u\right)\right)^{2} \d u\,\d t  \right]^{1/2} \\
	&\quad + \frac{6\beta}{N} \sum_{k=1}^{N} \int_{0}^{T} \Mean\left[ |(s^{N}_{k}(t-) - \bar{s}^{N}_{k}(t-)| + |\lambda^{N}_{k}(t-) - \bar{\lambda}^{N}_{k}(t-)| \right]\d t
\end{split}\\
\begin{split}
	&\leq \alpha \int_{0}^{T} \Mean\left[|\lambda^{N}_{i}(t-) - \bar{\lambda}^{N}_{i}(t-)|\right] \d t + \frac{8\beta\sqrt{2T}}{\sqrt{N}}\\
	&\quad + \frac{6\beta}{N} \sum_{k=1}^{N} \int_{0}^{T} \Mean\left[ |(s^{N}_{k}(t-) - \bar{s}^{N}_{k}(t-)| + |\lambda^{N}_{k}(t-) - \bar{\lambda}^{N}_{k}(t-)| \right]\d t.
\end{split}
\end{align*}
Since
\begin{align*}
	\Mean\left[\sup_{t\in [0,T]} |s^{N}_{i}(t) - \bar{s}^{N}_{i}(t)|\right] &\leq \Mean\left[\int_{0}^{T}\int_{(0,2)} \left|q\left((s^{N}_{i}(t-),\lambda^{N}_{i}(t-)),u\right) - q\left((\bar{s}^{N}_{i}(t-),\bar{\lambda}^{N}_{i}(t-)),u\right)\right| \d u\,\d t \right] \\
	&\leq 6\int_{0}^{T} \Mean\left[ |(s^{N}_{i}(t-) - \bar{s}^{N}_{i}(t-)| + |\lambda^{N}_{i}(t-) - \bar{\lambda}^{N}_{i}(t-)| \right]\d t,
\end{align*}
it follows that
\begin{align*}
	& \frac{1}{N}\sum_{i=1}^{N}\Mean\left[\sup_{t\in [0,T]} \left(|s^{N}_{i}(t) - \bar{s}^{N}_{i}(t)| + |\lambda^{N}_{i}(t) - \bar{\lambda}^{N}_{i}(t)|\right) \right] \\
	&\leq \frac{8\beta\sqrt{2T}}{\sqrt{N}} + \frac{\alpha + 6 + 6\beta}{N}\sum_{i=1}^{N} \int_{0}^{T} \Mean\left[|(s^{N}_{i}(t-) - \bar{s}^{N}_{i}(t-)| + |\lambda^{N}_{i}(t-) - \bar{\lambda}^{N}_{i}(t-)| \right]\d t \\
	&\leq \frac{8\beta\sqrt{2T}}{\sqrt{N}} + \left(\alpha + 6 + 6\beta\right) \int_{0}^{T} \frac{1}{N}\sum_{i=1}^{N}\Mean\left[\sup_{r\in[0,t]}\left(|(s^{N}_{i}(r) - \bar{s}^{N}_{i}(r)| + |\lambda^{N}_{i}(r) - \bar{\lambda}^{N}_{i}(r)|\right) \right]\d t.
\end{align*}
An application of Gronwall's lemma yields, for every $T > 0$,
\[
	\frac{1}{N}\sum_{i=1}^{N}\Mean\left[\sup_{t\in [0,T]} \left(|s^{N}_{i}(t) - \bar{s}^{N}_{i}(t)| + |\lambda^{N}_{i}(t) - \bar{\lambda}^{N}_{i}(t)|\right) \right] \stackrel{N\to\infty}{\longrightarrow} 0,
\]
which implies the desired convergence.

\end{appendix}

\section*{Acknowledgements}
 We are grateful to G. Giacomin for deep comments and suggestions.
The authors acknowledge the financial support of the Research Grant of the Ministero dell’Istruzione,
dell’Università e della Ricerca: PRIN 2009, Complex Stochastic Models and their Applications in Physics
and Social Sciences.

\bibliographystyle{plain}
\bibliography{tanh}

\end{document}